\documentclass[12pt]{article}

\usepackage{graphicx}
\usepackage{amsmath}
\usepackage{amssymb}
\usepackage{theorem}

\numberwithin{equation}{section}

\newtheorem{thm}{Theorem}[section]

\newtheorem{pro}[thm]{Proposition}

{\theorembodyfont{\rmfamily}
\newtheorem{defn}[thm]{Definition}
\newtheorem{examp}[thm]{Example}

}

\newcommand{\qed}{\hfill \mbox{\raggedright \rule{.07in}{.1in}}}
 
\newenvironment{proof}{\vspace{1ex}\noindent{\bf
Proof}\hspace{0.5em}}{\hfill\qed\vspace{1ex}}

\newcommand{\Z}{\mathbb{Z}}

\title{Deterministically Driven Random Walks on a Finite State Space.}
\author{Colin Little\thanks{Department of Mathematics, University of Surrey, Guildford, GU2 7XH, UK.}}

\begin{document}

\maketitle

\begin{abstract}
We introduce the concept of a \emph{deterministic walk}. Confining our attention to the finite state case, we establish hypotheses that ensure that the deterministic walk is \emph{transitive}, and show that this property is in some sense robust. We also establish conditions that ensure the existence of asymptotic occupation times. 
\end{abstract}

\section{Introduction}
Our purpose is to provide a foundational background for the study of deterministically driven random walks on a state space $S$. In particular, we are concerned to generalise properties of Markov chains to settings in which the dynamics governing the behaviour of the walk is deterministic, and where the Markov property does \textbf{not} necessarily hold. 

We define the deterministic walk as follows.

\begin{defn}\label{def:DetWalk} 
Let $T$ be a measurable transformation of a probability space $(X, m)$, and suppose that associated to each $i \in S$ is a measurable function $f_i: X \rightarrow S$. We call each $f_i$ a \emph{transition function}, and we call the collection $(f_k)_{k \in S}$ an \emph{environment} on $S$. 

Define the skew-product transformation $T_f: X \times S \rightarrow X \times S$ by
\begin{equation}\label{eq:skewprod}
T_f(x, i) := (Tx, f_i(x)).
\end{equation}
Define the \emph{deterministic walk on} $S$ \emph{in the environment} $(f_k)_{k \in S}$ to be
\begin{equation}\label{eq:detwalk}
U_n := U_{i,n}(x) := \pi_2(T_f^n(x, i))
\end{equation} 
where $\pi_2(x,y) := y$. 
\end{defn}
%We will often refer to the set $S$ as either the \emph{state space} or the \emph{fibre}. We will also refer to the probability space $(X,m)$ as the \emph{base}, and the map $T: X \rightarrow X$ as the \emph{base map}.

An important feature of the above is that random updates are highly dependent on the current position of the walk. While there is a considerable wealth of dynamical systems literature devoted to the study of systems in which transformations are applied randomly to a common state space, these predominantly deal with the situation where random updates are state-independent. A classic example of such systems is that of a \emph{random transformation} $R$ of a probability space $(Y, \nu)$. Given a finite set $\mathcal{T} := \{\tau_i: Y \rightarrow Y: i = 1,\ldots, K\}$ of measurable transformations of a probability space $(Y, \nu)$, and an i.i.d.\ sequence $(T_n)_{n \geq 1}$ of random variables taking values in $\mathcal{T}$, the random transformation $R$ of $(Y, \nu)$ is defined such that for all $n \geq 1$, $R^n := T_n \circ \ldots \circ T_1$. An equivalent, deterministic, representation of $R$ uses the skew-product $T_f: \Sigma^+ \times X \rightarrow \Sigma^+ \times X$ such that $T_f(\omega, x) := (\sigma(\omega), T_{\omega_0}(x))$, where $\Sigma := \{1, \ldots, K\}$, and $\sigma: \Sigma^+ \rightarrow \Sigma^+$ is the full-shift. In particular, for all $n \geq 1$, 
\begin{equation}\label{eq:PosIndepRmap}
R^n(x) = \pi_2(T_f^n(\omega, x)) = \tau_{\omega_{n-1}} \circ \ldots \circ \tau_{\omega_0}(x).
\end{equation}
The ergodic theory of such systems was studied extensively in \cite{Kifer1986}. Among the many areas in which the model given by (\ref{eq:PosIndepRmap}) (and generalisations thereof) play an important role are the theory of products of random matrices (which goes back to the work of Furstenberg and Kesten \cite{FurstKest1960}) and the theory of random iterated function systems (in which the properties of fractals generated by random applications of contractive maps to a complete metric space are studied - see  \cite{Barnsley1985, Hutchinson1981} for early developments in this field). For a relatively recent study of position dependent case see \cite{BBQ2008}.

While systems described by equation (\ref{eq:PosIndepRmap}) enjoy the Markov property, we are interested in understanding the asymptotic behaviour of deterministic walks that are non-Markov. 

Focusing on the finite state case, we establish recurrence properties under relatively mild assumptions. Notably, no mixing assumptions are required on the underlying dynamics. We also note that while distortion properties play an important role in the infinite state case (see \cite{Little2012b}), such assumptions are not required for establishing recurrence in the finite state case.

In Section \ref{sec:FiniteTransitivity}, we formally introduce the central notions of \emph{recurrence}, \emph{transience} and \emph{transitivity} of a deterministic walk, and motivate these ideas by considering simple examples. We then establish (in Theorem \ref{thm:GibbsMarkov_Type_Recurrence}) conditions for which a deterministic walk on a finite state space is transitive, and show that this property is in some sense robust.

In Section \ref{sec:OccupationTimes}, we establish (in Theorem \ref{thm:AsympOccTimes}) conditions under which a deterministic walk on a finite state space admits asymptotic occupation times.

\section{Transitivity of a Deterministic Walk on a Finite State Space}\label{sec:FiniteTransitivity}

The following properties of a deterministic walk are fundamental to the exposition of our main results.

\begin{defn}\label{def:RecTrans}
Given a deterministic walk on a state space $S$ in an environment $(f_k)_{k \in S}$, we say that a state $i \in S$ is \emph{recurrent} if 
\[
m(\{x \in X: U_{i,n}(x) = i, \mathrm{for~some}~n \geq 1\}) = 1. 
\]
We say that a state $i$ is \emph{transient} if it is not recurrent.

We say that the deterministic walk is \emph{transitive on} $S$ if for all $i, j \in S$, 
\[
m(\{x \in X: U_{i,n}(x) = j, \mathrm{for~some}~n \geq 1\}) = 1. 
\]
\end{defn}

\begin{pro}\label{pro:2StateTrans}
If $S$ consists of two states, $m(\{x \in X: f_i(x) = i\}) < 1$ for each $i \in S$, and $T$ is ergodic and measure-preserving, then the deterministic walk is transitive on $S$.
\end{pro}
\begin{proof}
Since $m(\{x \in X: f_i(x) = i\})< 1$ for each $i \in S$, and $T$ is ergodic and measure preserving, it follows from Birkhoff's Ergodic Theorem that $m(\cap_{n=0}^{\infty}T^{-n}\{x \in X: f_i(x) = i\}) = 0$. Hence for $i,j \in S$, $m(\{x: U_{i,n}(x) = j ~\mathrm{for~some}~n \geq 1\}) = 1$.
\end{proof}

\begin{examp}
In light of Proposition \ref{pro:2StateTrans} it might be hoped that if $S$ is finite, $T$ is ergodic and measure preserving, and $m(\{x \in X: f_i(x) = j\}) > 0$ for all $i, j \in S$, then $(U_n)_{n \geq 0}$ is necessarily transitive. As the following counter-example shows, this is not the case: Consider the deterministic walk on the state space $S = \{0, 1, 2\}$ driven by the transformation $T: [0,1] \rightarrow [0,1]$, where $Tx := 2x ~\mathrm{mod}~1$, and with transition functions $f_0$, $f_1$ and $f_2$ given by 
\begin{eqnarray*}
f_0(x) = \left\{ \begin{array}{llllllll} 
 0, x \in [\frac{3}{4},1) &\\
 1, x \in [\frac{1}{4}, \frac{3}{4}) &\\
 2, x \in [0 , \frac{1}{4})
\end{array} \right.
\end{eqnarray*}
\begin{eqnarray*}
f_1(x) = \left\{ \begin{array}{llllllll} 
 0, x \in [\frac{1}{4}, \frac{1}{2}) &\\
 1, x \in [0, \frac{1}{4}) \cup [\frac{3}{4}, 1] &\\
 2, x \in [\frac{1}{2}, \frac{3}{4}) &  
\end{array} \right.
\end{eqnarray*}
and
\begin{eqnarray*}
f_2(x) = \left\{ \begin{array}{llllllll} 
 0, x \in [0, \frac{1}{4}) \cup [\frac{1}{2}, \frac{3}{4}) &\\
 2, x \in [\frac{1}{4}, \frac{1}{2}) \cup [\frac{3}{4}, 1]. &  
\end{array} \right.
\end{eqnarray*}
Defining $\beta := \{[0, \frac{1}{2}), [\frac{1}{2}, 1)\}$, by inspection we see that the skew-product $T_f: [0,1] \times S \rightarrow [0,1] \times S$ decomposes the product state space $\beta \times S$ into two closed communication classes given by
$\{[0, \textstyle\frac{1}{2}) \times \{1\}~,~[\frac{1}{2}, 1] \times \{0\}\}$ 
and $\{[0, \textstyle\frac{1}{2}) \times \{0\}~,~[\textstyle\frac{1}{2}, 1] \times \{1\}~,~[0, \frac{1}{2}) \times \{2\}~,~[\textstyle\frac{1}{2}, 1] \times \{2\}\}$.

Since the first communication class does not project onto the whole fibre $S$, (in particular, the state 2 is not represented in this class) it follows that the deterministic walk is not transitive on $S$. In fact, it is easy to show that $T_f$ induces a Markov chain on the state space $\beta \times S$ (with respect to the product measure $\mathrm{Lebesgue} \times \mathrm{counting~measure}$), and it follows that any walk started in the second communication class necessarily visits all three states infinitely often.
\end{examp}

\begin{defn} Recall that a non-singular transformation $T$ of a measure space $(X, m)$ is said to be \emph{Markov}, with a measurable \emph{Markov partition} $\beta$ of $X$, if
\begin{itemize}
\item[(i)] for all $a \in \beta$, $T(a)$ is the union of elements of $\beta$,
\item[(ii)] $T_{|a}: a \rightarrow T(a)$ is a bijection.
\end{itemize}
Given a Markov map $T$ with with Markov partition $\beta$, for all $n \geq 1$ and all sequences $a_0, \ldots, a_{n-1} \in \beta$, we call the set $[a_0, \ldots, a_{n-1}] := \cap_{j=0}^{n-1}T^{-j}a_j$ a \emph{cylinder set of rank} $n$ or an $n$-\emph{cylinder}. 

For a given cylinder set $a$, we let $|a|$ denote its rank. 

We denote by $\beta_n$ the collection of all $n$-cylinders. 

We say that a cylinder set $a$ is \emph{admissible} if $m(a) > 0$.
\end{defn}

We introduce concepts for Markov maps that are analogous to related concepts for Markov chains. 

\begin{defn}\label{def:MarkovCommunication}
Let $T$ be a Markov transformation of a measure space $(X, m)$ with Markov partition $\beta$. Given $a, b \in \beta$ we say that $a$ \emph{communicates with} $b$ if
\[
m(\{x \in a: T^nx \in b~\mathrm{for~some}~n \geq 1\}) > 0.
\]
Given $a, b \in \beta$, we say that $a$ and $b$ \emph{intercommunicate} if $a$ communicates with $b$ and $b$ communicates with $a$.

Given $a \in \beta$ we define the \emph{communication class} of $a$ to be the set of all $b \in \beta$ with which it intercommunicates. (We observe that intercommunication is an equivalence relation.)

We say that a communication class $\mathcal{C}$ is \emph{closed} if for all $a \in \mathcal{C}$ and all $b \in \beta$, $a$ communicates with $b$, only if $b$ communicates with $a$. (When $\beta$ is finite, the existence of closed communication classes is automatic.)

We say that $\beta$ is \emph{irreducible} if $\beta$ is itself a communication class under $T$.
\end{defn}

\paragraph{Standing Hypotheses for the Deterministic Walk}
We assume that a deterministic walk on a finite state space $S$ is driven by a Markov transformation $T$ of a probability space $(X, m)$, with Markov partition $\beta$. In proving the main result of this section (Theorem \ref{thm:GibbsMarkov_Type_Recurrence} below) it is convenient to assume that transition functions are constant on elements of $\beta$, but as we observe later, this assumption is not entirely necessary.

\begin{defn}
Define the probability measure $\mu := \frac{m}{\#S} \times \mathrm{counting~measure}$.
\end{defn}
The following proposition is an easy consequence of the standing hypotheses. (The routine details are written out in \cite[Proposition 4.1]{Little2012a}.) 

Define $\tilde{\beta} := \beta \times S$.
\begin{pro}\label{pro:SkewMarkov}
The skew-product $T_f: X \times S \rightarrow X \times S$ is a Markov transformation of the probability space $(X \times S, \mu)$ with Markov partition $\tilde{\beta}$.
\end{pro}

\begin{thm}\label{thm:GibbsMarkov_Type_Recurrence}
Suppose that the base map $T$ is ergodic and invariant with finite Markov partition $\beta$. Then the deterministic walk is transitive on $S$ if and only if every closed communication class in $\tilde{\beta}$ projects onto the whole of $S$.
\end{thm}

\begin{proof}
By Proposition \ref{pro:SkewMarkov}, $T_f: X \times S \rightarrow X \times S$ is a Markov map with Markov partition $\tilde{\beta}$. We observe that the 'only if' direction of the proof is trivial, in that if there exists a closed communication class that does not project onto the whole of $S$ then, clearly, the deterministic walk cannot be transitive. In proving the 'if' part of the result we in fact prove something stronger.
\begin{itemize}
\item[(i)] $T_f$ is transitive on closed communication classes in $\tilde{\beta}$.
\item[(ii)] If a partition element $a \times \{i\} \in \tilde{\beta}$ does not lie in a closed communication class then $\mu\{(y, j) \in X \times S: T_f^n(y, j) \in a \times \{i\} ~\mathrm{i.o.}\} = 0$.
\end{itemize}

Since $S$ and $\beta$ are finite, so is $\tilde{\beta}$, and it follows from (i) and (ii) that with probability 1, an orbit under the skew-product dynamics must eventually end up in a closed communication class $\mathcal{C} \subset \tilde{\beta}$ on which the orbit is transitive thereafter. Hence, the deterministic walk is transitive on $S$ provided that every closed communication class projects onto the whole of $S$.\\
~\\
Proof of (i): Since $T$ is Markov with Markov partition $\beta$, and the environment of transition functions $(f_i)_{i \in S}$ are constant on elements of $\beta$, it follows that for all $n \geq 1$, every admissible $n$-cylinder $x := [x_0, \ldots, x_{n-1}]$, and all $r_0 \in S$, there is a unique itinerary on $S$
\[
r_0, r_1, \ldots , r_{n-1}
\]
such that $f_{r_i}(x_i) = r_{i+1}$ for $i = 0, \ldots, n-2$. We thus define
\[
t(x, r_0) := r_{n-1}.
\]
Fix $a \times \{j\}, b \times \{k\} \in \tilde{\beta}$ such that both lie in the same closed communication class $\mathcal{C}$. Let 
\begin{equation}\label{A}
a \times \{j_1\}, \ldots, a \times \{j_l\} \quad \& \quad b \times \{i_1\}, \ldots, b \times \{i_n\}
\end{equation}
be enumerations of all elements in $\mathcal{C}$ whose $\beta$-coordinates are $a$ and $b$ respectively. Since $T$ is ergodic and measure preserving, such collections are non-empty. We show that if we start in the communication class $\mathcal{C}$ then we must eventually visit $b \times \{k\}$.

Since $b \times \{i_1\}, \ldots, b \times \{i_n\}$ all lie in the closed communication class $\mathcal{C}$, then for each $s = 1, \ldots, n$ there exists a cylinder $w(i_s)$ of the form $[x_1, x_2, \ldots, b]$ such that
\begin{equation}\label{eq:Ws}
t([b, w(i_s)], i_s) := k,
\end{equation} 
and $[b, w(i_s)]$ is admissible (where in an abuse of notation we have written $[b, w(i_s)]$ as shorthand for $[b, x_1, x_2, \ldots, b]$). Equation (\ref{eq:Ws}) says that if the skew-product visits the set $[b, w(i_s)] \times \{i_s\}$, then it must eventually visit the partition element $b \times \{k\}$. 

Since $a \times \{j_1\}, b \times \{k\} \in \mathcal{C}$, there exists an admissible cylinder $E_1$ of the form $[a, \ldots, b]$ such that 
\[
t(E_1, j_1) = k.
\]
For $r = 1, \ldots, l-1$ we may inductively define admissible cylinders
\begin{equation}\label{eq:Er+1}
E_{r+1} = [E_r, w(t(E_r, j_{r+1}))].
\end{equation}
For each $r = 1, \ldots, l$ the cylinder $E_r$ is of the form $[a, \ldots, b]$. Thus, $t(E_r, j_{r+1})$ is the $S$-coordinate at time $|E_r|-1$ of every point in $E_r \times \{j_{r+1}\}$. Hence, every point in $E_r \times \{j_{r+1}\}$ visits the partition element $b \times \{t(E_r, j_{r+1})\}$ at time $|E_r| - 1$. It follows by construction that every point in the set $E_{r+1} \times \{j_{r+1}\}$ as defined in (\ref{eq:Er+1}) will visit the partition element $b \times \{k\}$ at time $|E_{r+1}|-1$. 

Since $m(E_l) > 0$ it follows from Birkhoff's Ergodic Theorem that the base dynamics must eventually visit $E_l$. Since we are in the closed communication class $\mathcal{C}$, and since $E_l \subset a$, it follows that when the base dynamics visit the set $E_l$, the deterministic walk visits one of the states $j_1, \ldots, j_l$. But since $E_l \subset E_{l-1} \subset \ldots \subset E_1$ it follows that we must eventually visit the partition element $b \times \{k\}$.  This completes the proof of (i).\\
~\\
Proof of (ii): Fix $a \times \{i\}$ lying in a communication class $\mathcal{C}$ that is not closed, and let $a \times i_1, \ldots, a \times i_l$ be all the other partition elements in $\mathcal{C}$ whose $\beta$-coordinate is $a$. Fix $b \times \{k\} \in \tilde{\beta}$ with which $a \times \{i\}$ communicates, but which does not communicate with $a \times \{i\}$. Let $A$ denote the set of partition elements that intercommunicate with $a \times \{i\}$ whose $\beta$-coordinate is $b$. Also, let $A'$ denote set of partition elements whose $\beta$-coordinate is $b$, with which $a \times \{i\}$ communicates, but that do not communicate with $a \times \{i\}$. Clearly, $b \times \{k\} \in A'$. 

By a similar argument to that given in the proof of part (i), for all $b \times \{j\} \in A$ there exists a cylinder $w(j) = [x_1, x_2, \ldots, b]$ such that $b \times \{t([b, w(j)], j)\} \in A'$.

Since $a \times \{i\}$ communicates with $b \times \{k\}$ there exists an admissible cylinder $E_0$ of the form $[a, \ldots, b]$ such that 
\[
t(E_0, i) = k.
\]
For $r = 0, \ldots, l-1$ we may inductively define admissible cylinders
\begin{eqnarray}\label{def:Edisp}
E_{r+1} := \left\{ \begin{array}{lllllllllll} 
E_r, ~\mathrm{if}~ b \times \{t(E_r, i_{r+1})\} \in A' &\\
{}[E_r, w(t(E_r, i_{r+1})) ], \mathrm{if}~b \times \{t(E_r, i_{r+1})\} \in A. &
\end{array} \right.
\end{eqnarray}
Let $a \times \{s_1\}, \ldots, a \times \{s_t\}$ be the set of all partition elements in $\tilde{\beta}$ whose $\beta$-coordinate is $a$, and with which $a \times \{i\}$ communicates, but which do not communicate with $a \times \{i\}$. 

Assuming that the skew-product dynamics start in $a \times \{i\}$ then by Birkhoff's Ergodic Theorem, the base dynamics must (with probability 1) eventually make a first visit to the set $E_l \subset a$, at which point, the deterministic walk must be in one of the states $i, i_1, \ldots, i_l, s_1, \ldots, s_t$. If it is one of the $s_i$ (for $i = 1, \ldots, t$), then we cannot return to $a \times \{i\}$ by assumption. If instead it is in one of the states $i, i_0, \ldots, i_l$ then, since $E_l \subseteq E_{l-1} \subseteq \ldots \subseteq E_0$, by the same argument as in part (i), it must either eventually visit the state $b \times \{k\}$ or some other element of the set $A'$. By definition, it cannot return from any of these partition elements to the partition element $a \times \{i\}$. This completes the proof of (ii).
\end{proof}

\paragraph{Robustness of transitivity} 
We call a set $A \subset X$ \emph{transitive} if for all $i \in S$ and all $x \in A$, the orbit of the associated deterministic walk $(U_{i,n}(x))_{n=0}^{\infty}$ visits every state in $S$ at least once. Under the hypotheses of Theorem \ref{thm:GibbsMarkov_Type_Recurrence}, it is immediate that if the deterministic walk is transitive on the state space $S$, then there must exist a transitive cylinder set $a$ of positive measure. Any perturbation $(f'_i)_{i \in S}$ of the environment $(f_i)_{i \in S}$ of transition functions that preserves the existence of a transitive positive measure subset $a' \subset a$ will, by Birkhoff's Ergodic Theorem, preserve the transitivity of the deterministic walk. Therefore, we may say that the transitivity of the deterministic walk on a finite fibre is \emph{robust} to perturbations of the environment. It follows that the hypothesis that transition functions are constant on elements of the Markov partition is not a necessary condition for the deterministic walk to be transitive on a finite state space.\\

\section{Asymptotic Occupation Times for a Deterministic Walk on a Finite State Space}\label{sec:OccupationTimes}

In this section, we continue to assume that the standing hypotheses hold and we establish conditions for the existence of asymptotic occupation times of a deterministic walk on a finite state space. 

\begin{defn}\label{def:StrongDistortion}

Let $T: X \rightarrow X$ be a Markov transformation of a probability space $(X, m)$, with Markov partition $\beta$. 
By the non-singularity of $T$, for all $n \geq 1$ and all $a \in \beta_n$, we may define the Radon-Nikodym derivative $v'_a := \frac{d(m \circ T^{-n})}{dm}$ such that for every measurable set $B$, $\int_{T^na \cap B} v'_a dm = m(a \cap T^{-n}B)$.

We say that $T$ has the \emph{Strong Distortion Property}, with distortion constant $D \geq 1$, if for all $n \geq 1$ and for all $a \in \beta_n$, and for a.e.\ $x, y \in T^n a$, $\frac{v'_a(x)}{v'_a(y)} \leq D$.
\end{defn}
The following proposition is an easy consequence of the standing hypotheses. (The routine details are written out in \cite[Proposition 4.3]{Little2012a}.)
\begin{pro}\label{pro:T_fisMGF}
Suppose that the map $T$ has the Strong Distortion Property. Then the skew-product $T_f: X \times S \rightarrow X \times S$ has the Strong Distortion Property with respect to the measure $\mu$.
\end{pro}

\begin{defn}
We say that a Markov transformation $T$ of probability space $(X, m)$, with Markov partition $\beta$, has \emph{Finite Images} if $\#\{Ta: a \in \beta\} < \infty$.
\end{defn}
The following result combines Lemma 4.4.1 and Theorem 4.6.3 in \cite{Aaronson1997}. (We note that \cite{Aaronson1997} uses the term \emph{topologically transitive} instead of \emph{irreducible}.)

\begin{pro}\label{pro:AaComb}
Let $T$ be an irreducible Markov transformation of a probability space $(X, m)$, with Markov partition $\beta$, satisfying Strong Distortion and Finite Images. Then there exists an ergodic, invariant, probability measure $m'$ equivalent to $m$.
\end{pro}
We may now state the main result of this section.

\begin{thm}\label{thm:AsympOccTimes}
Suppose that the state space $S$ is finite and that the transformation $T$ is a Markov map, with Markov partition $\beta$, that has Strong Distortion and Finite Images. Suppose also that $\tilde{\beta}$ is irreducible under $T_f$. Then there exists a distribution $(\pi_i)_{i \in S}$ such that for all $i, j \in S$
\begin{equation}\label{eq:ErgWalk}
\lim_{n \rightarrow \infty} \frac{1}{n}\sum_{r=0}^{n-1}\mathbf{1}_i(U_{j,r}(x)) = \pi_i \quad \mathrm{for}~m-\mathrm{a.e.}~ x.
\end{equation} 
\end{thm}
\begin{proof}
Since $T$ has Finite Images and since $S$ is finite, it follows that $T_f$ also has Finite Images. 

By Proposition \ref{pro:T_fisMGF}, $T_f$ has the Strong Distortion Property with respect to $\mu$. 

By assumption, $\tilde{\beta}$ is irreducible under $T_f$, and it follows from Proposition \ref{pro:AaComb} that $T_f$ has an ergodic, invariant, probability measure $\mu'$ equivalent to $\mu$. Thus, if for each $i \in S$ we define $\pi_i := \mu'(X \times \{i\})$, it follows from Birkhoff's Ergodic Theorem that
\begin{equation}\label{eq:ErgSkew}
\lim_{n \rightarrow \infty} \frac{1}{n}\sum_{r=0}^{n-1}\mathbf{1}_{i}(U_{j,r}(x)) = \pi_i \quad \mathrm{for}~\mu'-\mathrm{a.e.}~ (x, j).
\end{equation}
Equation (\ref{eq:ErgWalk}) now follows from (\ref{eq:ErgSkew}).
\end{proof}

\paragraph{Acknowledgements}
I am greatly indebted to Ian Melbourne for suggesting \emph{deterministically driven random walks} as a topic of research, and for his advice during the development this work. I would also like to thank the EPSRC for funding this research.

\end{document}